\documentclass[reqno,11pt]{amsart}
\usepackage{amssymb,amsmath,enumerate,tikz,hyperref}
\usepackage{bbm,dsfont,soul}

\numberwithin{equation}{section}
\newtheorem{theorem}[equation]{Theorem}
\newtheorem{lemma}[equation]{Lemma}
\newtheorem{proposition}[equation]{Proposition}
\newtheorem{corollary}[equation]{Corollary}

\theoremstyle{definition}
\newtheorem{definition}[equation]{Definition}
\newtheorem{example}[equation]{Example}
\newtheorem*{remark}{Remark}
\DeclareMathOperator{\rl}{rl}

\newcommand{\oeis}[1]{ \href{https://oeis.org/#1}{#1} }

\title[Rational Dyck paths and factor-free Dyck words]{On rational Dyck paths and the enumeration \\
of factor-free Dyck words}
\author{Daniel Birmajer}
\address{Department of Mathematics\\ Nazareth College\\ 4245 East Ave.\\ Rochester, NY 14618}
\author{Juan B. Gil}
\address{Penn State Altoona\\ 3000 Ivyside Park\\ Altoona, PA 16601}
\author{Michael D. Weiner}

\begin{document}
\maketitle

\begin{abstract}
Motivated by independent results of Bizley and Duchon, we study rational Dyck paths and their subset of factor-free elements. On the one hand, we give a bijection between rational Dyck paths and regular Dyck paths with ascents colored by factor-free words. This bijection leads to a new statistic based on the reducibility level of the paths for which we provide a corresponding formula. On the other hand, we prove an inverse relation for certain sequences defined via partial Bell polynomials, and we use it to derive a formula for the enumeration of factor-free words. In addition, we give alternative formulas for various enumerative sequences that appear in the context of rational Dyck paths.
\end{abstract}

\section{Introduction}
\label{sec:intro}

In his paper \cite{Bizley} of 1954, Bizley proved a formula (accredited to Howard Grossman) for the number $\phi_n$ of lattice paths from $(0,0)$ to $(\alpha n,\beta n)$ which may touch but never rise above the line $\alpha y=\beta x$, where $n$, $\alpha$, and $\beta$ are positive integers with $\gcd(\alpha,\beta)=1$. Such a path is called $\frac{\beta}{\alpha}$-Dyck path of length $(\alpha+\beta)n$. Using the notation $f_j=\frac1{(\alpha+\beta)j}\binom{(\alpha+\beta)j}{\alpha j}$ for $j\in\mathbb{N}$, Bizley's formula reads
\begin{equation} \label{eq:rationalPaths}
  \phi_n = \sum \frac{f_1^{k_1} f_2^{k_2}\cdots}{k_1!k_2!\cdots},
\end{equation}
where the sum runs over all $k_j\in\mathbb{N}_0$ such that $k_1+2k_2+\cdots+nk_n=n$. 

On a more recent paper, Duchon \cite{Duchon} studied generalized Dyck languages and discussed the particular case of a two-letter language with alphabet $\{a,b\}$, having valuations $h(a)=\beta$ and $h(b)=-\alpha$ with $\gcd(\alpha,\beta)=1$. In this case, associating the letter $a$ with the step $(1,0)$ and the letter $b$ with the step $(0,1)$, the set of all Dyck words (words with total valuation equal to 0) is in one-to-one correspondence with the set of $\frac{\beta}{\alpha}$-Dyck paths. 

In \cite[Theorem~9]{Duchon}, Duchon established a key connection between the set $\mathcal{D}_{\beta/\alpha}$ of generalized Dyck words with slope $\beta/\alpha$ and its subset of corresponding factor-free words,\footnote{A word in a language $L$ is said to be factor-free if it has no proper factor in $L$.} denoted by $\mathcal{\tilde D}_{\beta/\alpha}$. More precisely, he proved the identity 
\begin{equation}\label{eq:Thm9}
  \phi_n = [t^n] D_{\beta/\alpha}(t) 
  = \dfrac1{1+(\alpha+\beta)n}\, [t^n] \tilde{D}_{\beta/\alpha}^{1+(\alpha + \beta)n}(t),
\end{equation}
where $D_{\beta/\alpha}(t)$ and $\tilde{D}_{\beta/\alpha}(t)$ are the generating functions enumerating (by word length) the elements of $\mathcal{D}_{\beta/\alpha}$ and $\mathcal{\tilde D}_{\beta/\alpha}$, respectively. The identities \eqref{eq:rationalPaths} and \eqref{eq:Thm9} are the foundation for the results obtained in this paper. 

On the one hand, as a consequence of a result given in \cite[Theorem 3.5]{BGMW}, the right-hand side of \eqref{eq:Thm9} also counts the elements of $\mathfrak{D}^{\Theta}_n(\alpha+\beta,0)$, the set of Dyck words of semilength $(\alpha+\beta)n$ created from strings of the form $``d\,"$ and $``u^{(\alpha+\beta)j}d"$ for $j=1,\ldots,n$, such that each maximal ascent $u^{(\alpha+\beta)j}$ may be colored in as many different ways as the number $\theta_j$ of factor-free words of length $(\alpha+\beta)j$. In other words, 
\begin{quote}
there is a bijection between the set of $\frac{\beta}{\alpha}$-Dyck paths of length $(\alpha+\beta)n$ and the set of colored Dyck paths in $\mathfrak{D}^{\Theta}_n(\alpha+\beta,0)$.
\end{quote}
In Section~\ref{sec:bijection} we give an explicit bijection that reveals, in geometric terms, the factor-free factorization of a rational Dyck path. Specifically, we use the number of peaks of a path in $\mathfrak{D}^{\Theta}_n(\alpha+\beta,0)$ to define the {\em reducibility level} of the associated $\frac{\beta}{\alpha}$-Dyck path. As a application, we give a formula (Theorem~\ref{thm:Narayana-like}) to compute the number of $\frac{\beta}{\alpha}$-Dyck paths with a given reducibility level.

On the other hand, by rewriting \eqref{eq:rationalPaths} and \eqref{eq:Thm9} in terms of partial Bell polynomials, and using an interesting inverse relation that we prove in Proposition~\ref{prop:inverse}, 
\begin{quote}
we obtain a formula for the number $\theta_n$ of factor-free $\frac{\beta}{\alpha}$-Dyck words of length $(\alpha+\beta)n$ in terms of $f_1, f_2,\dots$ in the spirit of \eqref{eq:rationalPaths}.
\end{quote}
Our formula, given in Theorem~\ref{thm:factor-freeCount}, appears to be new for cases other than $\alpha=2$ and $\beta=3$. For the case of slope $3/2$, Duchon already observed (cf.\ \cite[Prop.~10]{Duchon}) that there are $\theta_n = C_n+C_{n-1}$ factor-free Dyck words of length $5n$, where $C_n$ is the $n$th Catalan number.

\medskip
We finish the paper with various formulas connecting the sequences $(f_n)$, $(\phi_n)$, and $(\theta_n)$ with the related sequence $(\psi_n)$ that counts the number of $\frac{\beta}{\alpha}$-Dyck paths of length $(\alpha+\beta)n$ that never touch the line $\alpha y=\beta x$ (except at the initial and terminal points).

\section{Rational Dyck paths as colored regular Dyck paths}
\label{sec:bijection}

In this paper, we will follow the terminology used in \cite{Duchon} for the study of generalized Dyck words. We consider the alphabet $U=\{a,b\}$ and assume the valuations $h(a)=\beta$ and $h(b)=-\alpha$ for positive integers $\alpha$ and $\beta$ with $\gcd(\alpha,\beta)=1$. A Dyck word $w$ with slope $\beta/\alpha$ is a string of letters from $U$ such that $h(w)=0$, and for each left factor $u$ of $w$, $h(u)\ge 0$. Let $\mathcal{D}_{\beta/\alpha}$ denote the set of all such words. Note that the length $|w|$ of a word in $\mathcal{D}_{\beta/\alpha}$ is always a multiple of $\alpha+\beta$.

A word $w'\in \mathcal{D}_{\beta/\alpha}$ is called a factor of $w$ if there are words $u$ and $v$ (possibly empty words) such that $w= uw'v$ and $uv\in \mathcal{D}_{\beta/\alpha}$. If $u$ and $v$ are both not empty words, then $w'$ is called a proper factor of $w$. A word $w\in \mathcal{D}_{\beta/\alpha}$ is said to be {\em factor-free} if it has no proper factors in $\mathcal{D}_{\beta/\alpha}$. Let $\varepsilon$ denote the empy word.

Let $\phi_n$ be the number of elements in $\mathcal{D}_{\beta/\alpha}$ of length $(\alpha+\beta)n$, and let $\theta_n$ be the number of factor-free words in $\mathcal{D}_{\beta/\alpha}$ of length $(\alpha+\beta)n$. 
Applying Fa{\`a} di Bruno's formula to the right-hand side of \eqref{eq:Thm9}, we get the equivalent representation
\begin{equation}\label{eq:yw}
  \phi_n = \sum_{k=1}^{n} \binom{(\alpha + \beta)n}{k-1} \frac{(k-1)!}{n!} B_{n, k}(1! \theta_1,2! \theta_2,\dots),
\end{equation}
where $B_{n,k}(x_1,\dots,x_{n-k+1})$ denotes the $(n,k)$-th partial Bell polynomial.

As shown in \cite[Theorem~3.5]{BGMW}, the right-hand side of \eqref{eq:yw} gives the number of regular Dyck paths of semilength $(\alpha+\beta)n$ whose ascents have length a multiple of $\alpha+\beta$ and such that an ascent of length $(\alpha+\beta)j$ may be colored in $\theta_j$ different ways. We denote this set of Dyck paths/words by $\mathfrak{D}^{\Theta}_n(\alpha+\beta,0)$. The purpose of this section is to provide an explicit bijection between elements in $\mathcal{D}_{\beta/\alpha}$ of length $(\alpha+\beta)n$ and elements of $\mathfrak{D}^{\Theta}_n(\alpha+\beta,0)$. To this end, we first introduce some notation and discuss one example.

Every word $w\not=\varepsilon$ in $\mathcal{D}_{\beta/\alpha}$ can be written, uniquely, as $w = uw'v$, where $w'$ is the left-most, nonempty factor-free subword of $w$. In this case, we say that the word $uv$ is a {\em reduction} of $w$ by $w'$ and write $w \underset{w'}{\rightarrow} uv$.

Conversely, if $w$ and $w'$ are Dyck words and $1\le j\le |w|$, then we can write $w=uv$ with $|u|=j$ and insert $w'$ between $u$ and $v$ to form the new word $\hat w=uw'v$. We call $\hat w$ the {\em extension} of $w$ by $w'$ at the position $j$ and write $\hat w = (w\underset{j}{\leftarrow}w')$.

\begin{example} \label{ex:bijection}
Consider the word $w=aabbabbaababaabbbbbb$ in $\mathcal{D}_{3/2}$, which represents the following $\frac32$-Dyck path of length 20:
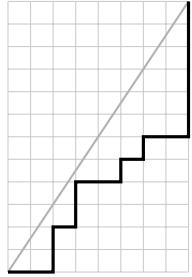
\begin{center}
\begin{tikzpicture}[scale=0.3]
\draw [step=1,thin,gray!40] (0,0) grid (8,12);
\draw [gray!60, thick] (0,0) -- (8,12);
\draw [very thick] (0,0) -- (2,0) -- (2,2) -- (3,2) -- (3,4) -- (5,4) -- (5,5) -- (6,5) -- (6,6) -- (8,6) -- (8,12);
\end{tikzpicture}
\end{center}
Factor $w$ as $w=u_1w'_1v_1=aabbabbaabab{\color{blue}(aabbb)}bbb$ noting that $w_1'={\color{blue}aabbb}$ is the left-most, nonempty factor-free subword of $w$. Thus the reduction $w \underset{w'_1}{\rightarrow} u_1v_1$ gives
\[ aabbabbaabab{\color{blue}(aabbb)}bbb \underset{\color{blue}aabbb}{\longrightarrow} aabbabbaababbbb. \]
Similarly, $w_2'={\color{blue}ababb}$ is factor-free, so we can reduce $aabbabba{\color{blue}(ababb)}bb \underset{\color{blue}ababb}{\longrightarrow} aabbabbabb$. Observe that $w_3'={\color{blue}aabbabbabb}$ is factor-free. If we let $\ell_j$ be the length of the factor $u_j$ in the $j$th reduction of $w$, then $\ell_1=|u_1| = |aabbabbaabab|=12$ and $\ell_2=|u_2|=|aabbabba|=8$.

Finally, we construct the colored Dyck path $D_w$ associated with $w$ as follows:
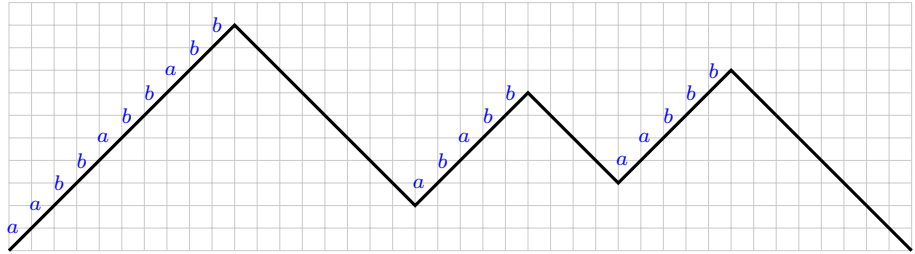
\begin{center}
\begin{tikzpicture}[scale=0.3]
\draw [step=1,thin,gray!40] (0,0) grid (40,11);
\draw [very thick] (0,0) -- (10,10) -- (18,2) -- (23,7) -- (27,3) -- (32,8) -- (40,0);
\color{blue}
\node [left=1pt] at (1,1) {\scriptsize $a$};
\node [left=1pt] at (2,2) {\scriptsize $a$};
\node [left=1pt] at (3,3) {\scriptsize $b$};
\node [left=1pt] at (4,4) {\scriptsize $b$};
\node [left=1pt] at (5,5) {\scriptsize $a$};
\node [left=1pt] at (6,6) {\scriptsize $b$};
\node [left=1pt] at (7,7) {\scriptsize $b$};
\node [left=1pt] at (8,8) {\scriptsize $a$};
\node [left=1pt] at (9,9) {\scriptsize $b$};
\node [left=1pt] at (10,10) {\scriptsize $b$};
\node [left=1pt] at (19,3) {\scriptsize $a$};
\node [left=1pt] at (20,4) {\scriptsize $b$};
\node [left=1pt] at (21,5) {\scriptsize $a$};
\node [left=1pt] at (22,6) {\scriptsize $b$};
\node [left=1pt] at (23,7) {\scriptsize $b$};
\node [left=1pt] at (28,4) {\scriptsize $a$};
\node [left=1pt] at (29,5) {\scriptsize $a$};
\node [left=1pt] at (30,6) {\scriptsize $b$};
\node [left=1pt] at (31,7) {\scriptsize $b$};
\node [left=1pt] at (32,8) {\scriptsize $b$};
\end{tikzpicture}
\end{center}
\[ D_w = u^{|w_3'|}d^{\ell_2} u^{|w_2'|} d^{\ell_1-\ell_2} u^{|w_1'|} d^{20-\ell_1} = u^{10}d^8 u^5 d^4 u^5 d^8, \]
where the ascents are colored (from left to right) by $aabbabbabb$, $ababb$, and $aabbb$.

The above construction is reversible. Note that for a given Dyck path $D$ of semilength 20, colored by $aabbabbabb$, $ababb$, and $aabbb$, one can create a unique word $w_D$ by successively inserting the factor-free words at positions determined by the downs of the Dyck path.
\end{example}

\begin{theorem} \label{thm:bijection}
Let $\alpha, \beta\in\mathbb{N}$ with $\gcd(\alpha,\beta)=1$. The algorithm outlined in Example~\ref{ex:bijection} provides a bijection between the set of $\frac{\beta}{\alpha}$-Dyck paths of length $(\alpha+\beta)n$ and $\mathfrak{D}^{\Theta}_n(\alpha+\beta,0)$.
\end{theorem}
\begin{proof}
Let $w$ be a word in $\mathcal{D}_{\beta/\alpha}$ of length $(\alpha+\beta)n$. Without loss of generality, we assume that the corresponding lattice path stays strongly below the line $y=\frac{\beta}{\alpha}x$ (except at the endpoints). For a general path, we just look at its connected components separately.

If $w$ is factor-free, then we define
\[ D_w = u^{(\alpha+\beta)n}d^{(\alpha+\beta)n} \]
and color the ascent with $w$. Clearly $D_w\in \mathfrak{D}^{\Theta}_n(\alpha+\beta,0)$.

If $w$ is not a factor-free word, then it can be factored uniquely as $w = u_1 w_1' v_1$, where $w_1'$ is a factor-free word in $\mathcal{D}_{\beta/\alpha}$, $h(u_1)>0$, and $h(v_1)<0$. Let $\ell_1 = |u_1|$ and let $u_1v_1$ be the reduction of $w$ by $w_1'$. If $u_1v_1$ is factor-free, then we denote it by $w_2'$ and define
\[ D_w= u^{|w_2'|} d^{\ell_1} u^{|w_1'|} d^{(\alpha+\beta)n-\ell_1}. \]
Otherwise, we write $u_1v_1$ as $u_2 w_2' v_2$, where $w_2'$ is the left-most factor-free word contained in $u_1v_1$. We let $\ell_2 = |u_2|$ and look at the reduction $u_2v_2$. We continue this process inductively until we get a factor-free reduction.

Suppose $w\in \mathcal{D}_{\beta/\alpha}$ is a word of length $(\alpha+\beta)n$ such that after $k-1$ reductions
\[ w \underset{w_1'}{\longrightarrow} u_1v_1 \underset{w_2'}{\longrightarrow} u_2v_2
   \underset{w_3'}{\longrightarrow} \cdots \underset{w_{k-1}'}{\longrightarrow} u_{k-1}v_{k-1} \]
we arrive at a factor-free word $w_k'=u_{k-1}v_{k-1}$. Let $\ell_j=|u_j|$ for $j\le k-1$, and define
\[ D_w= u^{|w_k'|} d^{\ell_{k-1}} u^{|w_{k-1}'|} d^{(\ell_{k-2}-\ell_{k-1})} \cdots
  u^{|w_2'|} d^{(\ell_1-\ell_2)} u^{|w_1'|} d^{(\alpha+\beta)n-\ell_1}. \]
By construction, there is a total of $(\alpha+\beta)n$ downs, $|w_1'|+\cdots+|w_k'|=(\alpha+\beta)n$, and for every $j=1,\dots,k-1$:
\begin{itemize}
\item[$\circ$] $|w_j'| \equiv 0 \!\mod (\alpha+\beta)$,
\item[$\circ$] $\ell_j-\ell_{j+1}\ge 1$ (letting $\ell_k=0$),
\item[$\circ$] $\ell_j\le |w_{j+1}'|+\dots+|w_{k}'|$.
\end{itemize}
In other words, $D_w$ represents a Dyck path of semilength $(\alpha+\beta)n$. Coloring each ascent $u^{|w_j'|}$ with the factor-free word $w_j'$, we get an element of $\mathfrak{D}^{\Theta}_n(\alpha+\beta,0)$.

The above algorithm can be easily reversed. Let $D$ be an element of $\mathfrak{D}^{\Theta}_n(\alpha+\beta,0)$ with $k$ peaks, whose ascents are colored by the factor-free words $w_1,\dots,w_k$. Thus $D$ must be of the form
\[ D = u^{|w_1|}d^{j_1}u^{|w_2|}d^{j_2} \cdots u^{|w_k|}d^{j_k}, \]
where each $|w_j|$ is a multiple of $\alpha+\beta$, $|w_1|+\dots+|w_k|=(\alpha+\beta)n$, and for every $i$ we have $|w_1|+\dots+|w_i|\ge j_1+\dots+j_i$. Once again, without loss of generality, we assume that $D$ stays strongly above the $x$-axis (it has no interior touch point).

We define $w_D$ by repeated insertion of the coloring factor-free words:
\[ w_D = \Big((\cdots((w_1 \underset{j_1}{\leftarrow} w_2) \underset{j_1+j_2}{\leftarrow} w_3)\cdots)
   \underset{j_1+\dots+j_{k-1}}{\longleftarrow} w_k\Big) \]
going through the ascents of $D$ from left to right. By construction, $w_D$ is clearly a Dyck word of length $(\alpha+\beta)n$ with slope $\beta/\alpha$. This finishes our bijection.
\end{proof}

\begin{definition}
Every word $w\in \mathcal{D}_{\beta/\alpha}$ has a maximal number of factor-free subwords by which the word can be reduced. We call the number of factor-free ``factors'' contained in $w$ the {\em reducibility level} of $w$, and denote it by $\rl(w)$. Thus $\rl(w)=1$ if and only if $w$ is a factor-free word. For example, for $w=aabbabbaababaabbbbbb$ as in Example~\ref{ex:bijection}, we have $\rl(w)=3$. We define $\rl(\varepsilon)=0$.
\end{definition}

\begin{proposition}[Integer slope]
If $\alpha=1$, then $ab^{\beta}$ is the only factor-free word in $\mathcal{D}_{\beta}$, and every word $w\in \mathcal{D}_{\beta}$ of length $(1+\beta)n$ has reducibility level $n$. Conversely, if $\beta>\alpha$ and $\rl(w)=n$ for every $w\in \mathcal{D}_{\beta/\alpha}$ with $|w|=(\alpha+\beta)n$, then $\alpha=1$.
\end{proposition}
\begin{proof}
Let $w\in\mathcal{D}_{\beta}$ be of length $(1+\beta)n$ with $n>1$. By definition, $h(a)=\beta$ and $h(b)=-1$. Thus the right-most $a$ in $w$ must be followed by a string containing $b^\beta$. Hence $w$ can be reduced by $ab^\beta$, so it is not factor-free and $\rl(w)=n$.

In order to prove the last statement, assume $\alpha>1$ and let $w=a^\alpha b^{\beta-1}aba^{\alpha-1}b^\beta$. Then we have $|w|=2(\alpha+\beta)$ and $\rl(w)=1$, which contradicts the assumption that $\rl(w)=2$ for every $w\in \mathcal{D}_{\beta/\alpha}$ with $|w|=2(\alpha+\beta)$.
\end{proof}

\begin{remark}
From the previous proposition, it is obvious that any reducibility level statistic is only meaningful for rational Dyck paths with non-integer slopes.
\end{remark}

\begin{theorem}\label{thm:Narayana-like}
The number $r_{n,k}$ of $\frac{\beta}{\alpha}$-Dyck paths of length $(\alpha+\beta)n$ that have reducibility level equal to $k$ is given by
\begin{equation} \label{eq:rl_stats}
  r_{n,k} = \binom{(\alpha + \beta)n}{k-1} \frac{(k-1)!}{n!} B_{n, k}(1! \theta_1,2! \theta_2,\dots),
\end{equation}
where $\theta_j$ is the number of factor-free words in $\mathcal{D}_{\beta/\alpha}$ of length $(\alpha+\beta)j$.
\end{theorem}
\begin{proof}
As a direct consequence of the bijection given in Theorem~\ref{thm:bijection}, we get that $r_{n,k}$ is equal to the number of Dyck paths in $\mathfrak{D}^{\Theta}_n(\alpha+\beta,0)$ having exactly $k$ peaks. Thus \eqref{eq:rl_stats} follows from \cite[Theorem 3.5]{BGMW}.
\end{proof}

\begin{example}
If $\alpha=1$, then $ab^{\beta}$ is the only factor-free word in $\mathcal{D}_{\beta}$, so $\theta_1=1$ and $\theta_j=0$ for $j\not=1$. Therefore, the number of words $w$ of length $(1+\beta)n$ with $\rl(w)=k$ is zero unless $k=n$. In that case \eqref{eq:rl_stats} gives
\begin{equation*}
  r_{n,n} = \binom{(1 + \beta)n}{n-1} \frac{(n-1)!}{n!} = \frac{1}{\beta n+1} \binom{(\beta+1)n}{n},
\end{equation*}
which is, as expected, the total number of words of length $(1+\beta)n$ in $\mathcal{D}_{\beta}$. 
\end{example}

\begin{example}
As mentioned in the introduction, for $\alpha=2$ and $\beta=3$, the number of factor-free words of length $5n$ is given by the sum of adjacent Catalan numbers $C_n+C_{n-1}$. Thus the number of $\frac32$-Dyck paths of length $5n$ having reducibility level $k$ is giving by
\begin{align*}
  r_{n,k} &= \binom{5n}{k-1}\frac{(k-1)!}{n!} B_{n,k}(1!(C_0+C_1), 2!(C_1+C_2),\dots),
\intertext{which by means of \cite[Example~3.2]{WW09} can be written as}
  r_{n,k} &= \binom{5n}{k-1} \sum_{j=0}^k \frac{(-1)^{k-j}(2j-k)}{nk} \binom{k}{j} \binom{2(n+j)-k-1}{n-1}.
\end{align*}
For example, among the 23 words of length 10 ($n=2$) there are $r_{2,1}=3$ factor-free words and $r_{2,2}=20$ words $w$ with $\rl(w)=2$. And among the 377 words of length 15 there are $7$ factor-free words, $90$ words with $\rl(w)=2$, and $280$ words with $\rl(w)=3$.
\end{example}

\section{An inverse relation involving partial Bell polynomials}
\label{sec:inverse_relation}

The purpose of this section is to prove an inverse relation for a family of sequences defined through partial Bell polynomials. The main ingredients are Fa{\`a} di Bruno's formula, expressed in terms of partial Bell polynomials as in \cite[Sec.~3.4, Theorem~A]{Comtet}, together with a substitution formula introduced by the authors in \cite{BGW12}. For convenience, we recall here these results as lemmas.

\begin{lemma}[Fa{\`a} di Bruno] \label{lem:FadiBruno}
Let $f$ and $g$ be two formal power series:
\begin{equation*}
  f = f_0 + \sum_{k=1}^\infty f_k \frac{u^k}{k!} \;\text{ and }\; g = \sum_{\ell=1}^\infty g_\ell \frac{t^\ell}{\ell!}.
\end{equation*}
If $h = \sum_{n=0}^\infty h_n \frac{t^n}{n!}$ is the formal power series of the composition $f\circ g$, then the coefficients $h_n$ are given by
\begin{equation*}
  h_0 = f_0, \quad  h_n = \sum_{k=1}^{n} f_k B_{n, k}(g_1,g_2,\dots).
\end{equation*}
\end{lemma}

\begin{lemma}[{\cite[Theorem 15]{BGW12}}] \label{lem:lambdaThm}
Let $a,b\in\mathbb{Z}$. Given any sequence $(x_n)$, define $(y_n)$ by
\begin{equation*}
y_n = \sum_{k=1}^n \binom{an+bk}{k-1}\frac{(k-1)!}{n!} B_{n,k}(1!x_1,2!x_2,\dots)
\end{equation*}
for every $n\in\mathbb{N}$. Then, for any $\lambda\in\mathbb{C}$ we have
\begin{equation*}
  \sum_{k=1}^n\tbinom{\lambda}{k-1}(k-1)!B_{n,k}(1!y_1,2!y_2,\dots)
  = \sum_{k=1}^n\tbinom{\lambda+an+bk}{k-1}(k-1)!B_{n,k}(1!x_1,2!x_2,\dots).
\end{equation*}
\end{lemma}

\begin{corollary}\label{cor:inverse}
If $y_n$ is given by
\begin{equation*}
y_n = \sum_{k=1}^n \binom{an}{k-1}\frac{(k-1)!}{n!} B_{n,k}(1!x_1,2!x_2,\dots),
\end{equation*}
then
\begin{equation*}
  x_n = \sum_{k=1}^n \binom{-an}{k-1}(k-1)!B_{n,k}(1!y_1,2!y_2,\dots).
\end{equation*}
\end{corollary}

\medskip
For the results in Section~\ref{sec:enumeration}, we also need the following inverse relation.

\begin{proposition}\label{prop:inverse}
Let $a\in\mathbb{Z}$. Given any sequence $(x_n)$, define $(y_n)$ by
\begin{equation*}
  y_n = \sum_{k=1}^n \binom{an-1}{k-1}\frac{(k-1)!}{n!}B_{n,k}(1!x_1,2!x_2,\dots).
\end{equation*}
Then, for every $n\in\mathbb{N}$ we have
\begin{equation*}
 x_n=\sum_{k=1}^n \frac{(1-an)^{k-1}}{n!} B_{n,k}(1!y_1,2!y_2,\dots).
\end{equation*}
\end{proposition}
\begin{proof}
Let $(z_n)$ be the sequence defined by
\begin{equation*}
 z_n = \sum_{k=1}^n \binom{an}{k-1}\frac{(k-1)!}{n!} B_{n,k}(1!x_1,2!x_2,\dots).
\end{equation*}
By Lemma~\ref{lem:lambdaThm} with $b=0$, we have
\begin{align} \label{eq:z(lambda-1)}
 \sum_{k=1}^n\tbinom{\lambda-1}{k-1}(k-1)!B_{n,k}(1!z_1,2!z_2,\dots)
 &= \sum_{k=1}^n\tbinom{\lambda+an-1}{k-1}(k-1)!B_{n,k}(1!x_1,2!x_2,\dots) \\
 \intertext{for any $\lambda\in\mathbb{C}$, and}
 \sum_{k=1}^n\tbinom{-1}{k-1}(k-1)!B_{n,k}(1!z_1,2!z_2,\dots) \label{eq:z(-1)}
 &= \sum_{k=1}^n\tbinom{an-1}{k-1}(k-1)!B_{n,k}(1!x_1,2!x_2,\dots).
\end{align}
If we denote $y(t)=\sum_{n=1}^\infty y_n t^n$ and  $z(t)=\sum_{n=1}^\infty z_n t^n$, then \eqref{eq:z(-1)} means (via Fa{\`a} di Bruno's formula) that the generating functions $y(t)$ and $z(t)$ are related by the identity
\[ y(t) = \log(1+z(t)). \]
Thus, if we let $f_1(t)=e^{\lambda t}$ and $f_2(t) = (1+t)^\lambda$, then
\[ f_1(y(t)) = f_1(\log(1+z(t))) = f_2(z(t)), \]
and Fa{\`a} di Bruno's formula gives
\begin{align*}
 [t^n] (f_1\circ y) &= \sum_{k=1}^n \lambda^{k} B_{n,k}(1!y_1,2!y_2,\dots), \text{ and }\\
 [t^n] (f_2\circ z) &= \sum_{k=1}^n (\lambda)_k\, B_{n,k}(1!z_1,2!z_2,\dots)
   = \sum_{k=1}^n \lambda\tbinom{\lambda-1}{k-1}(k-1)!B_{n,k}(1!z_1,2!z_2,\dots).
\end{align*}
This implies
\begin{equation*}
 \sum_{k=1}^n \lambda^{k} B_{n,k}(1!y_1,2!y_2,\dots)
 = \sum_{k=1}^n \lambda\tbinom{\lambda-1}{k-1}(k-1)!B_{n,k}(1!z_1,2!z_2,\dots).
\end{equation*}
Combining this identity with \eqref{eq:z(lambda-1)}, we obtain
\begin{equation*}
  \sum_{k=1}^n \lambda^{k-1} B_{n,k}(1!y_1,2!y_2,\dots)
  = \sum_{k=1}^n\tbinom{\lambda+an-1}{k-1}(k-1)!B_{n,k}(1!x_1,2!x_2,\dots).
\end{equation*}

Finally, if $\lambda=1-an$, then most terms on the right-hand side of the equation vanish and the sum reduces to $B_{n,1}(1!x_1,2!x_2,\dots)=n!x_n$. This gives the claimed identity.
\end{proof}

\section{Enumeration of factor-free Dyck words}
\label{sec:enumeration}

Let $\alpha$ and $\beta$ be positive integers with $\gcd(\alpha,\beta)=1$. As noted in the introduction, Bizley's formula \eqref{eq:rationalPaths} for the number of $\frac{\beta}{\alpha}$-Dyck paths of length $(\alpha+\beta)n$ can be conveniently written using partial Bell polynomials as
\begin{equation*}
  \phi_n=\sum_{k=1}^{n}\frac1{n!}B_{n,k}(1!f_1,2!f_2,\dots),
\end{equation*}
where $f_j=\frac1{(\alpha+\beta)j}\binom{(\alpha+\beta)j}{\alpha j}$ for $j\in\mathbb{N}$. Inverting the above identity, we get
\begin{equation}\label{eq:fnyn}
  f_n = \sum_{k=1}^{n} \binom{-1}{k-1}\frac{(k-1)!}{n!}B_{n,k}(1!\phi_1,2!\phi_2,\dots).
\end{equation}

On the other hand, as claimed in \eqref{eq:yw}, we have the alternative representation
\begin{equation}\label{eq:yw2}
  \phi_n = \sum_{k=1}^{n} \binom{(\alpha + \beta)n}{k-1} \frac{(k-1)!}{n!} B_{n, k}(1! \theta_1,2! \theta_2,\dots),
\end{equation}
where $\theta_j$ is the number of factor-free Dyck words with slope $\beta/\alpha$ and length $(\alpha+\beta)j$. Using Lemma~\ref{lem:lambdaThm} with $a=\alpha+\beta$, $b=0$, and $\lambda=-1$, we get
\begin{equation*}
  \sum_{k=1}^{n} \tbinom{-1}{k-1} (k-1)! B_{n,k}(1!\phi_1,2!\phi_2,\dots)
  = \sum_{k=1}^{n} \tbinom{(\alpha + \beta)n-1}{k-1} (k-1)! B_{n, k}(1! \theta_1,2! \theta_2,\dots),
\end{equation*}
which together with \eqref{eq:fnyn} gives the identity
\begin{equation}\label{eq:fnwn}
  f_n = \sum_{k=1}^{n} \binom{(\alpha + \beta)n-1}{k-1}\frac{(k-1)!}{n!} B_{n, k}(1! \theta_1,2! \theta_2,\dots)
\end{equation}
for every $n\in\mathbb{N}$.

\medskip
As a consequence of the inverse relation given in Proposition~\ref{prop:inverse}, we obtain:
\begin{theorem} \label{thm:factor-freeCount}
For every $n\in\mathbb{N}$, the number of factor-free Dyck words with slope $\beta/\alpha$ and length $(\alpha+\beta)n$ is given by
\begin{equation*}
  \theta_n = \sum_{k=1}^{n} \frac{\left(1-(\alpha+\beta)n\right)^{k-1}}{n!} B_{n, k}(1! f_1,2! f_2,\dots),
\end{equation*}
where $f_j=\frac1{(\alpha+\beta)j}\binom{(\alpha+\beta)j}{\alpha j}$ for every $j$.
\end{theorem}

This formula is easy to implement in any of the mainstream computer algebra systems. For example, in Sage, Maple, and Mathematica, partial Bell polynomials are implemented as {\tt bell\_polynomial}, {\tt IncompleteBellB}, and {\tt BellY}, respectively. The following table shows a few terms of the sequence $(\theta_n)$ for various slopes $\beta/\alpha$.

\begin{center}
\begin{tabular}{|c|c|l|} \hline
 $\beta/\alpha$ &OEIS & Sequence $(\theta_n)$ of factor-free words with slope $\beta/\alpha$ and length $(\alpha+\beta)n$\rule[-1.2ex]{0ex}{4ex} \\ \hline
 3/2 &\oeis{A005807}& 2, 3, 7, 19, 56, 174, 561, 1859, 6292, 21658, 75582, 266798, \dots \rule[-1ex]{0ex}{4.2ex} \\[5pt]
 5/2 &\oeis{A274052}& 3, 13, 94, 810, 7667, 76998, 805560, 8684533, 95800850, 1076159466, \dots \\[5pt]
 7/2 &\oeis{A274244}& 4, 34, 494, 8615, 165550, 3380923, 71999763, 1580990725, 35537491360, \dots \\[5pt]
 9/2 &\oeis{A274256}& 5, 70, 1696, 49493, 1593861, 54591225, 1950653202, 71889214644, \dots \\[5pt]
 4/3 &\oeis{A274257}& 5, 52, 880, 17856, 399296, 9491008, 235274240, 6014201600, \dots \\[5pt]
 5/3 &\oeis{A274258}& 7, 133, 4140, 154938, 6398717, 281086555, 12882897819, 609038885805, \dots \\[5pt]
 7/3 &\oeis{A274259}& 12, 570, 44689, 4223479, 441010458, 49014411306, 5685822210429, \dots \\[3pt] \hline
\end{tabular}
\end{center}

\medskip
The other class of lattice paths considered by Bizley \cite{Bizley} is the set of $\frac{\beta}{\alpha}$-Dyck paths that stay {\em strongly} below the line $y=\frac{\beta}{\alpha}x$. He proved that the number $\psi_n$ of such paths of length $(\alpha+\beta)n$ is given by
\begin{equation*}
  \psi_n=\sum_{k=1}^{n}\frac{(-1)^{k-1}}{n!}B_{n,k}(1!f_1,2!f_2,\dots).
\end{equation*}

Written in terms of $(\psi_n)$, the sequences $(f_n)$, $(\phi_n)$, and $(\theta_n)$ show an interesting pattern.
\begin{proposition} \label{prop:psi}
The following identities hold for every $n\in\mathbb{N}$:
\begin{align} \label{eq:f_of_psi}
 f_n &= \sum_{k=1}^{n}\frac{(k-1)!}{n!}B_{n,k}(1!\psi_1,2!\psi_2,\dots), \\ \label{eq:invert_psi}
 \phi_n &= \sum_{k=1}^{n}\frac{k!}{n!}B_{n,k}(1!\psi_1,2!\psi_2,\dots), \\ \label{eq:w_of_psi}
 \theta_n &= \sum_{k=1}^{n} \binom{-(\alpha+\beta)n+k}{k-1} \frac{(k-1)!}{n!}B_{n,k}(1!\psi_1,2!\psi_2,\dots).
\end{align}
\end{proposition}

\medskip
Note that the above expressions are instances of
\begin{equation*}
  \sum_{k=1}^{n} \binom{r+k-1}{k-1}\frac{(k-1)!}{n!}B_{n,k}(1!\psi_1,2!\psi_2,\dots)
\end{equation*}
for $r=0, 1, 1-(\alpha+\beta)n$, respectively. It is worth mentioning that the sequence $(\phi_n)$ is the {\sc invert} transform of the sequence $(\psi_n)$.

\begin{proof}[Proof of Proposition~\ref{prop:psi}]
Let $f(t)$, $\phi(t)$, and $\psi(t)$ be the generating functions of $(f_n)$, $(\phi_n)$, and $(\psi_n)$, respectively. Since $\psi(t) = 1 - e^{-f(t)}$ and $1+\phi(t) = e^{f(t)}$, we get
\[ f(t)=-\log(1-\psi(t)) \,\text{ and }\, 1+\phi(t) = \frac{1}{1-\psi(t)}, \]
which give identities \eqref{eq:f_of_psi} and \eqref{eq:w_of_psi} via Fa{\`a} di Bruno's formula.

Using \eqref{eq:invert_psi} we now write $\phi_n$ as
\begin{equation*}
 \phi_n = \sum_{k=1}^{n}\binom{k}{k-1}\frac{(k-1)!}{n!}B_{n,k}(1!\psi_1,2!\psi_2,\dots)
\end{equation*}
and use Lemma~\ref{lem:lambdaThm} with $a=0$, $b=1$, and $\lambda=-(\alpha+\beta)n$ to conclude that
\begin{equation*}
  \sum_{k=1}^n\tbinom{-(\alpha+\beta)n}{k-1}\tfrac{(k-1)!}{n!} B_{n,k}(1!\phi_1,2!\phi_2,\dots)
  = \sum_{k=1}^n\tbinom{-(\alpha+\beta)n+k}{k-1}\tfrac{(k-1)!}{n!} B_{n,k}(1!\psi_1,2!\psi_2,\dots).
\end{equation*}
Finally, because of the representation \eqref{eq:yw2}, Corollary~\ref{cor:inverse} implies that the left-hand side of the above identity is precisely $\theta_n$. Thus \eqref{eq:w_of_psi} holds.
\end{proof}


\end{document}